\documentclass[11pt,twoside]{article}
\usepackage{acta-m,euler}

\newtheorem{conjecture}[proposition]{Conjecture}

\title{Generalized perfect numbers}

\newcommand{\vol}{\textbf{1}, 1 (2009) 73--82}   
\newcommand{\amss}{\amssubj{
11A25, 11Y70
}}   
\newcommand{\keyw}{\keywords{
perfect number, superperfect number, $k$-hyperperfect number
}}

\begin{document}
\maketitle


\twoauthors{
Antal Bege
}{
Sapientia--Hungarian University of Transilvania\\
Department of Mathematics and Informatics,\\
T\^argu Mure\c{s}, Romania
}{abege@ms.sapientia.ro
}{
Kinga Fogarasi
}
{
Sapientia--Hungarian University of Transilvania\\
Department of Mathematics and Informatics,\\
T\^argu Mure\c{s}, Romania
}{kinga@ms.sapientia.ro
}



\short{
A. Bege, K. Fogarasi
}{
Generalized perfect numbers
%
}

\begin{abstract}
Let $\sigma(n)$ denote the sum of positive divisors of the natural number $n$.
A natural number is perfect if $\sigma(n) = 2n$. This concept was already generalized in form of superperfect
numbers $\sigma^2(n) = \sigma (\sigma (n)) = 2n$ and hyperperfect numbers $\sigma(n) = \frac{k+1}{k} n + \frac{k-1}{k}.$
\\
In this paper some new ways of generalizing perfect numbers are investigated, numerical results are presented and some conjectures are
established.
\end{abstract}

\section{Introduction}

For the natural number $n$ we denote the sum of positive divisors by
\[
\sigma(n)=\sum\limits_{d\mid n} d.
\]

\begin{definition}
A positive integer $n$ is called perfect number if it is equal to the sum of its proper divisors. Equivalently:
\[
\sigma (n) = 2n,
\]
where
\end{definition}

\begin{example}
The first few perfect numbers are: $6, 28, 496, 8128, \dots$ (Sloane's A000396 \cite{11}), since
\begin{eqnarray}
6 &=& 1 + 2 + 3 \nonumber\\
28 &=& 1 + 2 + 4 + 7 + 14\nonumber\\
496 &=& 1 + 2 + 4 + 8 + 16 + 31 + 62 + 124 + 248\nonumber
\end{eqnarray}
Euclid discovered that the first four perfect numbers are generated by the formula $2^{n-1} (2^n -1)$. He also noticed that $2^n-1$ is a prime number for every instance, and in Proposition IX.36 of "Elements" gave the proof, that the discovered formula gives an even perfect number whenever $2^n-1$ is prime.
\\
Several wrong assumptions were made, based on the four known perfect numbers:

\begin{itemize}

\item[$\bullet$] Since the formula $2^{n-1} (2^n-1)$ gives the first four perfect numbers for $n = 2, 3, 5,$ and 7 respectively,
the fifth perfect number would be obtained when $n = 11$. However $2^{11} - 1 = 23 \cdot 89$ is not prime, therefore this doesn't yield a
perfect number.

\item[$\bullet$] The fifth perfect number would have five digits, since the first four had 1, 2, 3, and 4 digits respectively, but it has 8 digits. The perfect
numbers would alternately end in 6 or 8.

\item[$\bullet$] The fifth perfect number indeed ends with a 6, but the sixth also ends in a 6, therefore the alternation is disturbed.

\end{itemize}
\end{example}

    In order for $2^n-1$ to be a prime, $n$ must itself to be a prime.

\begin{definition}
A \textbf{Mersenne  prime} is a prime number of the form:
\[
M_n = 2^{p_n} - 1
\]
where $p_n$ must also be a prime number.
\end{definition}

Perfect numbers are intimately connected with these primes, since there is a concrete one-to-one association between \emph{even} perfect numbers and Mersenne primes. The fact that Euclid's formula gives all possible even perfect numbers was proved by Euler two millennia after the formula was discovered.
\\
Only 46 Mersenne primes are known by now (November, 2008 \cite{9}), which means there are 46 known even perfect numbers. There is a conjecture that there are infinitely many perfect numbers. The search for new ones is the goal of a distributed search program via the Internet, named GIMPS (Great Internet Mersenne Prime Search) in which hundreds of volunteers use their personal computers to perform pieces of the search.
\\
It is not known if any \emph{odd} perfect numbers exist, although numbers up to $10^{300}$ (R. Brent, G. Cohen, H. J. J. te Riele \cite{brent1}) have been checked without success. There is also a distributed searching system for this issue of which the goal is to increase the lower bound beyond the limit above.
Despite this lack of knowledge, various results have been obtained concerning the odd perfect numbers:

\begin{itemize}
\item[$\bullet$] Any odd perfect number must be of the form $12 m + 1$ or $36m + 9$.

\item[$\bullet$] If $n$ is an odd perfect number, it has the following form:
\[
n = q^\alpha p_1^{2e_1} \dots p_k^{2e_k},
\]
where $q, p_1, \dots, p_k$
are distinct primes and $q \equiv \alpha \equiv 1 \pmod{4}$. (see L. E. Dickson \cite{dickson1})

\item[$\bullet$] In the above factorization, $k$ is at least 8, and if 3 does not divide $N$, then $k$ is at least 11.

\item[$\bullet$] The largest prime factor of odd perfect number $n$ is greater than $10^8$
(see T. Goto, Y. Ohno \cite{goto1}), the second largest prime factor is greater than $10^4$ (see D. Ianucci \cite{ianucci1}),
and the third one is greater than $10^2$ (see D. Iannucci \cite{ianucci2}).

\item[$\bullet$] If any odd perfect numbers exist in form
\[
n = q^\alpha p_1^{2e_1} \dots p_k^{2e_k},
\]
they would have at least 75 prime factor in total, that means: $\alpha + 2 \sum\limits_{i=1}^k e_i \ge 75.$ (see K. G. Hare \cite{4})
\end{itemize}

D. Suryanarayana introduced the notion of superperfect number in 1969 \cite{8}, here is the definition.
\begin{definition}
A positive integer $n$ is called \textbf{superperfect number} if
\[
\sigma (\sigma (n)) = 2n.
\]
\end{definition}

Some properties concerning superperfect numbers:

\begin{itemize}
\item[$\bullet$] Even superperfect numbers are $2^{p-1}$, where $2^p -1$ is a Mersenne prime.

\item[$\bullet$] If any odd superperfect numbers exist, they are square numbers (G. G. Dandapat \cite{dandapat1}) and either $n$ or $\sigma(n)$ is divisible by at least three distinct primes. (see H. J.  Kanold \cite{5})
\end{itemize}

\section{Hyperperfect numbers}

Minoli and Bear \cite{minoli1} introduced the concept of $k$-hyperperfect number and they conjecture that there are $k$-hyperperfect numbers for every $k$.

\begin{definition}
A positive integer $n$ is called \textbf{$k$-hyperperfect number} if
\[
n = 1 + k[\sigma (n) - n-1]
\]
rearranging gives:
\[
\sigma (n) = \frac{k+1}{k} n + \frac{k-1}{k}.
\]
\end{definition}

\noindent
We remark that a number is perfect iff it is 1-hyperperfect.
In the paper of J. S. Craine \cite{6} all hyperperfect numbers less than $10^{11}$ have been computed

\begin{example}
The table below shows some $k$-hyperperfect numbers for different $k$ values:
\\

\begin{center}
\begin{tabular}{|c|l|}
 \hline
  $\textbf{k}$ & $\mathbf{k}$-\textbf{hyperperfect} number    \\ \hline
  1 & 6 ,28, 496, 8128, ... \\
  2 & 21, 2133, 19521, 176661, ... \\
  3 & 325, ... \\
  4 & 1950625, 1220640625, ... \\
  6 & 301, 16513, 60110701, ... \\
  10 & 159841, ... \\
  12 & 697, 2041, 1570153, 62722153, ... \\
  \hline
\end{tabular}
\end{center}

\end{example}

\noindent
Some results concerning hyperperfect numbers:
\begin{itemize}
\item[$\bullet$] If $k > 1$ is an odd integer and $p = (3k+1)/2$ and $q = 3k + 4$ are prime numbers, then $p^2q$ is $k$-hyperperfect;
J. S. McCraine \cite{6} has conjectured in 2000 that all $k$-hyperperfect numbers for odd $k > 1$ are of this form, but the hypothesis has not been
proven so far.

\item[$\bullet$] If $p$ and $q$ are distinct odd primes such that $k (p+q) = pq - 1$ for some integer, $k$ then $n = pq$ is $k$-hyperperfect.

\item[$\bullet$] If $k > 0$ and $p = k+1$ is prime, then for all $i > 1$ such that $q = p^i - p +1$ is prime, $n = p^{i-1} q$ is $k$-hyperperfect
(see H. J. J. te Riele \cite{teriele1}, J. C. M. Nash \cite{7}).
\end{itemize}

We have proposed some other forms of generalization, different from $k$-hyperperfect numbers, and also we have examined \textbf{super-hyperperfect numbers} ("super" in the way as super perfect):

\begin{eqnarray}
&& \sigma (\sigma (n)) = \frac{k+1}{k} n + \frac{k-1}{k} \nonumber \\
&& \sigma (n) = \frac{2k-1}{k} n + \frac{1}{k} \nonumber \\
&& \sigma (\sigma (n)) = \frac{2k-1}{k} n + \frac{1}{k} \nonumber \\
&& \sigma (n) = \frac{3}{2} (n+1) \nonumber \\
&& \sigma (\sigma (n)) = \frac{3}{2} (n+1) \nonumber
\end{eqnarray}

\section{Numerical results}

For finding the numerical results for the above equalities we have used the ANSI C programming language, the Maple and the Octave programs.
Small programs written in C were very useful for going through the smaller numbers up to $10^7$, and for the rest we used the two other
programs. In this chapter the small numerical results are presented only in the cases where solutions were found.

3.1. Super-hyperperfect numbers. The table below shows the results we have reached:

\begin{table}[htbp]
  \centering
\begin{tabular}{|c|l|}
 \hline
  $\textbf{k}$ & \hspace{2cm} \textbf{n}   \\ \hline
  1 & $2, 2^2, 2^4, 2^6, 2^{12}, 2^{16}, 2^{18}$ \\
  2 & $3^2, 3^6, 3^{12}$ \\
  4 & $5^2$ \\
   \hline
\end{tabular}
\end{table}

3.2. $\sigma(n) = \frac{2k-1}{k} n + \frac{1}{k}$

For $k = 2:$

\begin{table}[htbp]
  \centering
\begin{tabular}{|c|l|}
 \hline
  $\textbf{n}$ & \textbf{prime factorization}   \\ \hline
  21 & $3 \cdot 7 = 3(3^2 - 2)$ \\
  2133 & $3^3 \cdot 79 = 3^3 \cdot (3^4 - 2)$ \\
  19521 & $3^4 \cdot 241 = 3^4 \cdot (3^5 - 2)$ \\
  176661 & $3^5 \cdot 727 = 3^5 \cdot (3^6 - 2)$ \\
   \hline

\end{tabular}
\end{table}

We have performed searches for $k = 3$  and $k = 5$ too, but we haven't found any solution

3.3. $\sigma (\sigma (n)) = \frac{2k-1}{k} n + \frac{1}{k}$

For $k=2:$

\begin{table}[htbp]
  \centering
\begin{tabular}{|c|l|}
 \hline
  $\textbf{k}$ & \textbf{prime factorization}  \\ \hline
  9 & $3^2$ \\
  729 & $3^6$ \\
  531441 & $3^{12}$ \\
   \hline
\end{tabular}
\end{table}

We have performed searches for $k = 3$ and $k = 5$ too, but we haven't found any solution

3.4. $\sigma (n) = \frac{3}{2} (n+1)$

\begin{table}[htbp]
  \centering
\begin{tabular}{|c|l|}
 \hline
  $\textbf{k}$ & \textbf{prime factorization}  \\ \hline
  15 & $3\cdot 5$ \\
  207 & $3^2 \cdot 23$ \\
  1023 & $3\cdot 11 \cdot 31$ \\
  2975 & $5^2 \cdot 7 \cdot 17$ \\
  19359 & $3^4 \cdot 239$ \\
  147455 & $5\cdot 7 \cdot 11 \cdot 383$ \\
  1207359 & $3^3 \cdot 97 \cdot 461$ \\
  5017599 & $3^3 \cdot 83 \cdot 2239$\\
   \hline
\end{tabular}
\end{table}

\section{Results and conjectures}

\begin{proposition}
 If $n = 3^{k-1} (3^k - 2)$ where $3^k - 2$ is prime, then $n$ is a 2-hyperperfect number.
\end{proposition}

\begin{proof}
Since the divisor function $\sigma$ is multiplicative and for a prime $p$ and prime power we have:
\[
\sigma (p) = p+1
\]
and
\[
\sigma (p^\alpha) = \frac{p^{\alpha +1} - 1}{p-1},
\]
it follows that:
\begin{eqnarray}
\sigma (n) &=& \sigma (3^{k-1} (3^k - 2)) = \sigma (3^{k-1}) \cdot \sigma (3^k - 2) = \frac{3^{(k-1)+1} -1}{3-1} \cdot (3^k - 2+1) = \nonumber\\
&=& \frac{(3^k - 1) \cdot (3^k -1 )}{2} = \frac{3^{2k} - 2\cdot 3^k + 1}{2} = \frac{3}{2} 3^{k-1} (3^k - 2) + \frac{1}{2}.\nonumber
\end{eqnarray}
\end{proof}

\begin{conjecture}
All 2-hyperperfect numbers are of the form $n = 3^{k-1} (3^k - 2),$ where $3^k - 2$ is prime.
\end{conjecture}

\noindent
We were looking for adequate results fulfilling the suspects, therefore we have searched for primes that can be written as $3^k - 2$. We have reached the following results:

\begin{center}
\begin{tabular}{|c|c|}
 \hline
 \textbf{ \# }& $k$ \textbf{ for which }$3^k - 2$ \textbf{is prime} \\ \hline
 1 & 2 \\
 2 & 4 \\
 3 & 5\\
 4 & 6 \\
 5 & 9 \\
 6 & 22 \\
 7 & 37 \\
 8 & 41 \\
 9 & 90 \\
  \hline
\end{tabular}
\end{center}

\begin{center}
\begin{tabular}{|c|c|}
 \hline
 \textbf{ \# }& $k$ \textbf{ for which }$3^k - 2$ \textbf{is prime} \\ \hline
 10 & 102 \\
 11 & 105 \\
 12 & 317 \\
 13 & 520 \\
 14 & 541 \\
 15 & 561 \\
 16 & 648 \\
 17 & 780 \\
 18 & 786 \\
 19 & 957 \\
 20 & 1353 \\
 21 & 2224 \\
 22 & 2521 \\
 23 & 6184 \\
 24 & 7989 \\
 25 & 8890 \\
 26 & 19217 \\
 27 & 20746 \\
   \hline
\end{tabular}
\end{center}

\noindent
Therefore the last result we reached is: $3^{20745} (3^{20746} - 2)$, which has 19796 digits.

\noindent
If we consider the super-hiperperfect numbers in special form
$\sigma (\sigma (n)) = \frac{3}{2} n + \frac{1}{2}$
we prove the following result.

\begin{proposition}
If $n = 3^{p-1}$ where $p$ and $(3^p -1)/2$ are primes, then $n$ is a super-hyperperfect number.
\end{proposition}

\begin{proof}
\begin{eqnarray*}
\sigma (\sigma(n)) &=& \sigma (\sigma (3^{p-1})) = \sigma \left(\frac{3^p -1}{2} \right) = \frac{3^p -1}{2} + 1 =\\
&=&  \frac{3}{2} \cdot 3^{p-1} + \frac{1}{2}=\frac{3}{2}n+\frac{1}{2}.
\end{eqnarray*}
\end{proof}

\begin{conjecture}
All solutions for this generalization are $3^{p-1}$-like numbers, where $p$ and $(3^p  -1)/2$ are primes.
\end{conjecture}

We were looking for adequate results fulfilling the suspects, therefore we have searched for primes $p$  for which $(3^p-1)/2$ is also prime.
We have reached the following results:

\begin{center}
\begin{tabular}{|c|c|}
 \hline
 \textbf{ \# }& $p-1$ \textbf{for which}$p$ and $(3^p - 1)/2$ \textbf{are primes} \\ \hline
 1 & 2 \\
 2 & 6 \\
 3 & 12\\
 4 & 540 \\
 5 & 1090 \\
 6 & 1626 \\
 7 & 4176 \\
 8 & 9010 \\
 9 & 9550 \\
    \hline
\end{tabular}
\end{center}

\noindent
Therefore the last result we reached is: $3^{9550}$, which has 4556 digits.



\bigskip
\rightline{\emph{Received: November 9, 2008}}     

\end{document}